\documentclass[a4paper]{amsart}

\usepackage{fixltx2e}
\usepackage[T1]{fontenc}
\usepackage{amssymb,amsthm,amsmath}
\usepackage[british]{babel}
\usepackage{xcolor}
\usepackage[all]{xy}
\usepackage{calrsfs}
\usepackage{graphicx}
\usepackage{hyperref}
\usepackage{mathabx}

\theoremstyle{plain}
\newtheorem{theorem}[subsection]{Theorem}
\newtheorem{lemma}[subsection]{Lemma}
\newtheorem{proposition}[subsection]{Proposition}
\newtheorem{corollary}[subsection]{Corollary}

\theoremstyle{definition}
\newtheorem{definition}[subsection]{Definition}
\newtheorem{example}[subsection]{Example}
\newtheorem{remark}[subsection]{Remark}

\newenvironment{tfae}
{
\begin{enumerate}}
{\end{enumerate}}

\newcommand{\comp}{\raisebox{0.2mm}{\ensuremath{\scriptstyle{\circ}}}}
\newcommand{\defn}{\textbf}
\newcommand{\noproof}{\hfill \qed}
\newcommand{\To}{\Rightarrow}

\newcommand{\gp}{\mathrm{Gp}}
\newcommand{\mon}{\mathrm{Mon}}

\newcommand{\rng}{\mathrm{Rng}}
\newcommand{\grd}{\mathrm{Grd}}
\newcommand{\norm}{\ensuremath{\mathrm{Norm}}}

\newcommand{\free}{\ensuremath{\mathrm{F}}}

\newcommand{\normal}{\ensuremath{\mathrm{N}}}

\newcommand{\triv}{\ensuremath{\mathrm{Triv}}}
\newcommand{\cod}{\ensuremath{\mathrm{cod}}}

 \DeclareMathOperator{\ab}{Ab}

\DeclareMathOperator{\Eq}{Eq}
\newcommand{\R}[1]{\Eq(#1)}

\newcommand{\C}{\ensuremath{\mathbb{C}}}
\newcommand{\E}{\ensuremath{\mathcal{E}}}
\newcommand{\F}{\ensuremath{\mathcal{F}}}
\newcommand{\X}{\ensuremath{\mathbb{X}}}
\newcommand{\s}{\ensuremath{\mathcal{S}}}

\newcommand{\Mon}{\ensuremath{\mathsf{Mon}}}
\newcommand{\Grp}{\ensuremath{\mathsf{Gp}}}

\newcommand{\Ab}{\ensuremath{\mathsf{Ab}}}
\newcommand{\SRng}{\ensuremath{\mathsf{SRng}}}
\newcommand{\Rng}{\ensuremath{\mathsf{Rng}}}

\newcommand{\Arr}{\ensuremath{\mathsf{Arr}}}

\newcommand{\Fib}{\ensuremath{\mathsf{Fib}}}

\newcommand{\Norm}{\ensuremath{\mathsf{Norm}}}

\newcommand{\Pt}{\ensuremath{\mathsf{Pt}}}
\newcommand{\SPt}{\ensuremath{\mathsf{SPt}}}

\newcommand{\Triv}{\ensuremath{\mathsf{Triv}}}

\def\pullback{
 \ar@{-}[]+R+<6pt,-1pt>;[]+RD+<6pt,-6pt>%
 \ar@{-}[]+D+<1pt,-6pt>;[]+RD+<6pt,-6pt>}

\def\halfsplitpullback{%
 \ar@{-}[]+R+<6pt,-1pt>;[]+RD+<6pt,-6pt>%
 \ar@{-}[]+D+<.5ex,-6pt>;[]+RD+<6pt,-6pt>}

\def\ophalfsplitpullback{%
 \ar@{-}[]+R+<6pt,-1pt>;[]+RD+<6pt,-6pt>%
 \ar@{-}[]+D+<.5ex,-6pt>;[]+RD+<6pt,-6pt>}

\def\splitpullback{%
 \ar@{-}[]+R+<6pt,-.5ex>;[]+RD+<6pt,-6pt>%
 \ar@{-}[]+D+<1ex,-6pt>;[]+RD+<6pt,-6pt>}

\def\splitsplitpullback{%
 \ar@{-}[]+R+<6pt,-.5ex>;[]+RD+<6pt,-6pt>%
 \ar@{-}[]+D+<.5ex,-6pt>;[]+RD+<6pt,-6pt>}

\def\skewpullback{%
 \ar@{-}[]+RU+<3pt,.5ex>;[]+R+<15pt,0pt>%
 \ar@{-}[]+RD+<3pt,-3.5pt>;[]+R+<15pt,0pt>}

\def\cubepullback{%
 \ar@{-}[]+RD+<6pt,-6pt>;[]+RDD+<6pt,-18pt>%
 \ar@{-}[]+D+<.5ex,-7pt>;[]+RDD+<6pt,-18pt>}

\def\cubepullbackplus{%
 \ar@{-}[]+RD+<-8pt,-8pt>;[]+RDD+<-8pt,-22pt>%
 \ar@{-}[]+D+<.5ex,-6pt>;[]+RDD+<-8pt,-22pt>}

 \def\cubepullbackminus{%
 \ar@{-}[]+RD+<-2pt,-4pt>;[]+RDD+<-2pt,-18pt>%
 \ar@{-}[]+D+<.5ex,-6pt>;[]+RDD+<-2pt,-18pt>}

 \def\cubepullbackdots{%
 \ar@{.}[]+RD+<6pt,-6pt>;[]+RDD+<6pt,-18pt>%
 \ar@{.}[]+D+<.5ex,-7pt>;[]+RDD+<6pt,-18pt>}

\def\sskewpullback{%
 \ar@{-}[]+RU+<6pt,1pt>;[]+R+<17pt,0pt>%
 \ar@{-}[]+RD+<6pt,-3pt>;[]+R+<17pt,0pt>}

 \def\pushout{%
 \ar@{-}[]+L+<-6pt,1pt>;[]+LU+<-6pt,6pt>%
 \ar@{-}[]+U+<-1pt,6pt>;[]+LU+<-6pt,6pt>}

\newdir{>>}{{}*!/3.5pt/:(1,-.2)@^{>}*!/3.5pt/:(1,+.2)@_{>}*!/7pt/:(1,-.2)@^{>}*!/7pt/:(1,+.2)@_{>}}
\newdir{ >>}{{}*!/8pt/@{|}*!/3.5pt/:(1,-.2)@^{>}*!/3.5pt/:(1,+.2)@_{>}}
\newdir{ |>}{{}*!/-3.5pt/@{|}*!/-8pt/:(1,-.2)@^{>}*!/-8pt/:(1,+.2)@_{>}}
\newdir{ >}{{}*!/-8pt/@{>}}
\newdir{>}{{}*:(1,-.2)@^{>}*:(1,+.2)@_{>}}
\newdir{<}{{}*:(1,+.2)@^{<}*:(1,-.2)@_{<}}

\begin{document}

\title[Reflectiveness of normal extensions]{A criterion for\\ reflectiveness of normal extensions}

\author{Andrea Montoli}
\address[Andrea Montoli]{CMUC, Universidade de
Coimbra, 3001--501 Coimbra, Portugal\newline and\newline Institut
de Recherche en Math\'ematique et Physique, Universit\'e
catholique de Louvain, chemin du cyclotron~2 bte~L7.01.02, 1348
Louvain-la-Neuve, Belgium}
\thanks{This work was partially supported by
the Centre for Mathematics of the
 University of Coimbra -- UID/MAT/00324/2013, funded by the Portuguese
 Government through FCT/MEC and co-funded by the European Regional Development Fund
 through the Partnership Agreement PT2020.} \email{montoli@mat.uc.pt}

\author{Diana Rodelo}
\address[Diana Rodelo]{CMUC, Universidade de Coimbra,
3001--501 Coimbra, Portugal\newline and\newline Departamento de
Matem\'atica, Faculdade de Ci\^{e}ncias e Tecnologia, Universidade
do Algarve, Campus de Gambelas, 8005--139 Faro, Portugal}
\thanks{The first author is
a Postdoctoral Researcher of the Fonds de la Recherche
Scientifique--FNRS} \email{drodelo@ualg.pt}

\author{Tim Van~der Linden}
\address[Tim Van~der Linden]{Institut de
Recherche en Math\'ematique et Physique, Universit\'e catholique
de Louvain, che\-min du cyclotron~2 bte~L7.01.02, B--1348
Louvain-la-Neuve, Belgium}
\thanks{The third author is a Research
Associate of the Fonds de la Recherche Scientifique--FNRS}
\email{tim.vanderlinden@uclouvain.be}



\begin{abstract}
We give a new sufficient condition for the normal extensions in an
admissible Galois structure to be reflective. We then show that
this condition is indeed fulfilled when $\X$ is the (protomodular)
reflective subcategory of $\s $-special objects of a Barr-exact
$\s $-protomodular category $\C$, where $\s $ is the class of
split epimorphic trivial extensions in $\C $. Next to some
concrete examples where the criterion may be applied, we also
study the adjunction between a Barr-exact unital category and its
abelian core, which we prove to be admissible.
\end{abstract}

\maketitle

\section{Introduction}

In the paper~\cite{MRVdL} we studied the adjunction between the
category of monoids and the category of groups, given by the group
completion of a monoid, from the point of view of categorical
Galois theory. We showed that the adjunction is admissible with
respect to the class of surjective homomorphisms, and we described
the central extensions (which turn out to coincide with the normal
extensions): they are the so-called \emph{special homogeneous
surjections} (see \cite{SchreierBook}). In the subsequent
paper~\cite{MRVdL2}, we showed that special homogeneous
surjections of monoids are reflective amongst surjective
homomorphisms. In order to do so, we applied Theorem 4.2
in~\cite{JK:Reflectiveness}.

The adjunction between monoids and groups is an instance of a more
general situation, recently described in~\cite{SchreierBook} and
in~\cite{S-proto}: the category of monoids is \emph{$\s
$-protomodular}, with respect to a suitable class $\s $ of points
(=~split epimorphisms with a fixed splitting), and the category of
groups is its \emph{protomodular core} relatively to the class $\s
$ (see Section~\ref{Section S-protomodular categories}). $\s
$-protomodularity allows us to recover, for monoids, relative
versions of several important properties of
Mal'tsev~\cite{Carboni-Kelly-Pedicchio} and
protomodular~\cite{Bourn1991} categories, like the Split Short
Five Lemma, or the fact that every internal reflexive relation is
transitive.

The case of monoids and groups now suggests the following general
question: given an adjunction, admissible with respect to regular
epimorphisms, between a category with ``weak'' algebraic
properties and a reflective subcategory with ``strong''
properties, like a protomodular one, such that the big category is
$\s $-protomodular with respect to the class $\s $ of split
epimorphic trivial extensions, is it always the case that normal
extensions are reflective amongst regular epimorphisms?

The present paper gives an affirmative answer to this question for
the case of Barr-exact categories~\cite{Barr}. In order to do
this, we needed to obtain a new criterion for reflectiveness of
normal extensions, Theorem~\ref{Theorem Reflectiveness General}:
given a Galois structure between Barr-exact categories, which is
admissible with respect to classes of regular epimorphisms, the
category of normal extensions is reflective in the category of all
fibrations (as the morphisms in the chosen class of regular
epimorphisms are called) provided that it is closed under
coequalizers of reflexive graphs.

The paper is organised as follows. In Section~\ref{general
results} we recall some basic notions of categorical Galois theory
and we prove our criterion for reflectiveness of normal
extensions. In Section~\ref{Section S-protomodular categories} we
recall the definition, some properties and some examples of $\s
$-protomodular categories. Section~\ref{Section S-application} is
devoted to the proof that the criterion can be applied in the
context of Barr-exact $\s $-protomodular categories. In
Section~\ref{Examples} we describe the concrete examples of the
adjunction between monoids and groups and the one between
semirings and rings. Section~\ref{Section additive core} is
devoted to the study of a general class of examples, namely the
adjunction between a Barr-exact unital~\cite{B0} category and its
\emph{abelian core}. In particular, we prove that, for any
finitely cocomplete Barr-exact unital category, the reflection to
its abelian core gives an admissible Galois structure, and that
the criterion for reflectiveness of normal extensions is
applicable to this Galois structure.


\section{Reflectiveness of normal extensions}\label{general results}

In this section we work towards a general result on reflectiveness
of normal extensions in an admissible Galois structure:
Theorem~\ref{Theorem Reflectiveness General} which says that, if
the fibrations in the Galois structure are regular epimorphisms,
and normal extensions are closed under coequalisers of reflexive
graphs, then the normal extensions are reflective amongst the
fibrations.

\subsection{Galois structures}
\label{Subsection: Galois structures} We begin by recalling the
notion of an \emph{(admissible) Galois structure} as well as the
concepts of \emph{trivial}, \emph{normal} and \emph{central
extension} arising from it~\cite{Janelidze:Pure,
Janelidze:Recent,Janelidze-Kelly}. We consider the context of
Barr-exact categories~\cite{Barr} and restrict ourselves to
fibrations which are regular epimorphisms to avoid some technical
difficulties.

\begin{definition} \label{Def: Galois structure basic}
A \defn{Galois structure} \(\Gamma=(\C,\X,I,H,\eta, \epsilon, \E,
\F)\) consists of an adjunction
\[
\xymatrix{\C \ar@<1ex>[r]^I \ar@{}[r]|\bot & \X \ar@<1ex>[l]^H}
\]
with unit \(\eta\colon {1_{\C}\To HI}\) and counit
\(\epsilon\colon {IH\To 1_{\X}}\) between Barr-exact categories
\(\C\) and~\(\X\), as well as classes of morphisms \(\E\) in
\(\C\) and \(\F\) in \(\X\) such that:
\begin{enumerate}
 \item[(G1)] \(\E\) and \(\F\) contain all isomorphisms;
 \item[(G2)] \(\E\) and \(\F\) are pullback-stable;
 \item[(G3)] \(\E\) and \(\F\) are closed under composition;
 \item[(G4)] \(H(\F)\subseteq \E\);
 \item[(G5)] \(I(\E)\subseteq \F\).
\end{enumerate}
We call the morphisms in \(\E\) and \(\F\)
\defn{fibrations}~\cite{Janelidze:Recent}. We moreover assume
\begin{enumerate}
\item[(G6)] the classes $\E $ and $\F $ consist of the regular
epimorphisms in $\C$ and in $\X$, respectively.
\end{enumerate}
Finally, we assume that $\C$ has coequalisers of reflexive graphs.
\end{definition}

The following definitions are given with respect to a Galois
structure $\Gamma $.

\begin{definition}
\label{trivial, central, normal}
 A \defn{trivial extension} is a fibration \(f\colon{A\to B}\) in $\C $ such
 that the square
 \[
 \xymatrix{A \pullback \ar[r]^-{\eta_A} \ar[d]_f & HI(A) \ar[d]^{HI(f)}\\
 B \ar[r]_-{\eta_B} & HI(B)}
 \]
is a pullback. A \defn{central extension} is a fibration \(f\)
whose pullback \(p^*(f)\) along \emph{some} fibration \(p\) is a
trivial extension. A \defn{normal extension} is a fibration such
that its kernel pair projections are trivial extensions.
\end{definition}

It is easy to see that trivial extensions are always central
extensions and that any normal extension is necessarily a central
extension.

Given any object \(B\) in \(\C\), we can associate an adjunction
\[
 \xymatrix{(\E\downarrow B) \ar@<1ex>[r]^-{I^B} \ar@{}[r]|-\bot & (\F\downarrow I(B)), \ar@<1ex>[l]^-{H^B}}
\]
where \((\E\downarrow B)\) denotes the full subcategory of the
slice category \((\C\downarrow B)\) determined by the morphisms in
\(\E\); similarly for \({(\F\downarrow I(B))}\). The functor
\(I^B\) is just the restriction of \(I\), while \(H^B\) sends a
fibration \(g\colon {X\to I(B)}\) to the pullback
\[
\xymatrix{A \ar[r] \ar[d]_-{H^B(g)} \pullback & H(X) \ar[d]^-{H(g)}\\
B \ar[r]_-{\eta_B} & HI(B)}
\]
of~\(H(g)\) along \(\eta_B\).

\begin{definition} \label{Def: Galois structure}
A Galois structure \(\Gamma=(\C,\X,I,H,\eta,\epsilon,\E,\F)\) is
said to be
\defn{admissible} when, for every object \(B\) in \(\C\), the
functor \(H^B\) is full and faithful.
\end{definition}

In the presence of an admissible Galois structure, every trivial
extension is always a normal extension:

\begin{proposition}[\cite{JK:Reflectiveness}, Proposition~2.4]\label{Proposition Trivial implies normal}
If \(\Gamma\) is an admissible Galois structure, then \(I\colon
{\C\to \X}\) preserves pullbacks along trivial extensions. Hence a
fibration is a trivial extension if and only if it is a pullback
of some fibration in $H(\X) $. In particular, the trivial
extensions are pullback-stable, so that every trivial extension is
a normal extension.\noproof
\end{proposition}

The admissibility condition of a Galois structure together with
the proposition above give the needed conditions to have the
reflectiveness of trivial extensions amongst fibrations. In fact,
the replete image of the functor $H^{B}$ is the category of
trivial extensions over $B$, denoted by $\Triv(B) $. Moreover,
$\Triv(B) $ is a reflective subcategory of \((\E\downarrow B)\),
where $H^{B}I^{B}\colon (\E\downarrow B)\to \Triv(B) $ is its
reflector. So, by Proposition~5.8 in~\cite{Im-Kelly}, we obtain a
left adjoint, called the \defn{trivialisation functor}
\[
\triv\colon{\Fib(\C)\to \Triv(\C),}
\]
to the inclusion of the category $\Triv(\C) $ of trivial
extensions in $\C $ into the full subcategory $\Fib(\C) $ of the
category of arrows in $\C $ determined by the fibrations.

\subsection{Reflectiveness of normal extensions}
Given an admissible Galois structure $\Gamma $ as in
Definition~\ref{Def: Galois structure} and an object $B$ in $\C
$, we denote by $\Norm(B) $ the full subcategory of
\((\E\downarrow B)\) determined by the normal extensions over $B$.
When it exists, the left adjoint to the inclusion functor
${\Norm(B)\hookrightarrow(\E\downarrow B)}$ will be denoted by
$\norm\colon{(\E\downarrow B) \to \Norm(B)}$ and called the
\defn{normalisation functor (over $B$)}. We also write
\[
\norm\colon \Fib(\C) \to \Norm(\C)
\]
for the left adjoint to the inclusion
$\Norm(\C)\hookrightarrow\Fib(\C) $ (where $\Norm(\C) $ is the
category whose objects are the normal extensions in $\C $) which
exists as soon as the normalisation functors over all objects $B$
exist (again by Proposition~5.8 in~\cite{Im-Kelly}, using that
normal extensions are stable under pullback).

We use the construction proposed in~\cite{EverMaltsev} and prove
that it does indeed provide us with a normalisation functor as
soon as the Galois structure $\Gamma $ is admissible and satisfies
the following condition:
\begin{enumerate}
\item[(G7)] $\Norm(\C) $ is closed under coequalisers of reflexive
graphs in $\Fib(\C)$.
\end{enumerate}
This approach is related to the results in~\cite{ED-NormalExt}
where the problem of reflectiveness of normal extensions is
studied in a much more general setting. Our present paper
and~\cite{ED-NormalExt} were written independently and around the
same, but with a different purpose in mind. Ours was to provide
simple applications of the construction in~\ref{construction}
below---essentially a simple version of the one proposed
in~\cite{Bourn-Rodelo2}, which strictly speaking cannot be applied
in the current context.

\subsection{The construction}\label{construction}
Given a fibration $f\colon{A\to B}$, we pull it back along itself,
then we take kernel pairs vertically as on the left hand side of
the diagram in Figure~\ref{Figure normalisation}. We apply the
trivialisation functor to obtain the upper right part of the
diagram, then we take the coequaliser $\underline{f}$ on the right
to get the morphism $\norm(f) $ and the comparison
$\eta_{f}^{\norm}$. The normality of $\norm(f) $ comes from
condition (G7) and the fact that all trivial extensions are normal
extensions (Proposition~\ref{Proposition Trivial implies normal}).
\begin{figure}[h]
\(\xymatrix@=30pt{ \Eq(\pi_{2})
\ar@/^1em/@{-->}[rr]^-{\eta_{\pi'_{1}}^{\triv }}
\ar@{->}@<-1ex>[d] \ar@{->}@<1ex>[d] \ar@{->>}[r]_-{\pi'_{1}} &
\R{f} \ar@<-1ex>[d] \ar@<1ex>[d] &
 \Eq(\pi_{2})_{\triv } \ar@{->>}[l]^-{\triv (\pi'_{1})} \ar@<-1ex>[d] \ar@<1ex>[d] \\
\R{f} \ar@/_1em/@{-->}[rr] \pullback \ar@{->}[u]
\ar@{->>}[d]_-{\pi_{2}} \ar@{->>}[r]^-{\pi_{1}} & A \ar[u]
\ar@{->>}[d]_-{f} & \R{f}_{\triv }
\ar@{->>}[l]_-{\triv (\pi_{1})} \ar[u] \ar@{->>}[d]^-{\underline{f}} \\
A \ar@{->>}[r]^-{f} \ar@/_1em/@{-->}[rr]_-{\eta_{f}^{\norm}} & B &
\overline{A} \ar@{->>}[l]_-{\norm(f)}}\) \caption{The construction
of $\norm(f) $}\label{Figure normalisation}
\end{figure}

\subsection{The universal property}
Let us prove that the extension $\norm(f) $ is universal amongst
all normal extensions over $B$. Suppose that $f=g\comp \alpha $,
where $g\colon {C\to B}$ is a normal extension. First note that
all steps of the construction are functorial. Next, since $g $ is
a normal extension, we have $\norm(g)=g $, $\overline{C}=C$ and
${\eta_{g}^{\norm}=1_C}$. So we get an induced morphism
$\overline{\alpha}\colon {\overline{A}\to C}$ such that $g\comp
\overline{\alpha}=\norm(f) $ and $\overline{\alpha}\comp
\eta_{f}^{\norm}=\alpha $, which proves the existence of a
factorisation. Now for the uniqueness, suppose that $\beta $,
$\gamma\colon \overline{A}\to C$ are such that
\[
g\comp \beta=\norm(f)=g\comp \gamma \quad\text{and}\quad
\beta\comp \eta_{f}^{\norm}=\alpha=\gamma\comp \eta_{f}^{\norm}.
\]
We write $\pi_1^f,\pi_2^f $ and $\pi_1^g,\pi_2^g $ for the kernel
pair projections of $f $ and $g $, respectively. From the fact
that $g $ is a normal extension, we have $\triv(\pi_1^g)=\pi_1^g $
and $\underline{g}=\pi_2^g $. Since $g\comp \alpha\comp \triv
(\pi^{f}_{1})=f\comp \triv (\pi^{f}_{1})=\norm(f)\comp
\underline{f}=g\comp \beta\comp \underline{f}$ and, likewise,
$g\comp \alpha\comp \triv (\pi^{f}_{1})=g\comp \gamma\comp
\underline{f}$, we find morphisms
\[
\widetilde{\beta}=\langle\alpha\comp
\triv(\pi_{1}^{f}),\beta\comp\underline{f}\rangle ,
\widetilde{\gamma}=\langle\alpha\comp
\triv(\pi_{1}^{f}),\gamma\comp\underline{f}\rangle\colon
{\R{f}_{\triv }\to \R{g}}
\]
such that $\pi_2^g\comp \widetilde{\beta}=\beta\comp
\underline{f}$ and $\pi_2^g\comp \widetilde{\gamma}=\gamma\comp
\underline{f}$ while
\[
 \pi^{g}_{1}\comp \widetilde{\beta}=\alpha\comp \triv (\pi^{f}_{1})
\qquad\text{and}\qquad \pi^{g}_{1}\comp
\widetilde{\gamma}=\alpha\comp \triv (\pi^{f}_{1}).
\]
Now $\widetilde{\beta}=\widetilde{\gamma}$ follows from the
uniqueness in the universal property of the trivial extension
$\triv (\pi_{1}^{f}) $: indeed, $\widetilde{\beta}\comp
\eta^{\triv
}_{\pi^{f}_{1}}=\alpha\times_{1_{B}}\alpha=\widetilde{\gamma}\comp
\eta^{\triv }_{\pi^{f}_{1}}$. Hence $\beta=\gamma $.

\subsection{The result}
Thus, keeping Proposition~5.8 in~\cite{Im-Kelly} in mind, we
obtain:
\begin{theorem}\label{Theorem Reflectiveness General}
Let \(\Gamma=(\C,\X,I,H,\eta, \epsilon, \E, \F)\) be an admissible
Galois structure such that the conditions {\rm (G6)} and {\rm
(G7)} hold. For any object $B$ in~$\C $, $\Norm(B) $ is a
reflective subcategory of~${(\E\downarrow B)}$. As a consequence,
normal extensions are reflective amongst fibrations.\noproof
\end{theorem}

\subsection{A weaker condition}
Condition (G7) is nice and simple, but it is slightly too strong
to be applied to $\s$-protomodular categories as in
Section~\ref{Section S-application}. We may replace it by the
following slightly weaker alternative, which is clearly still
strong enough to imply the conclusion of Theorem~\ref{Theorem
Reflectiveness General}:

\begin{enumerate}
\item[(G7${}^{-}$)] $\Norm(\C) $ is closed under coequalisers, in
the category $\Arr(\C)$ of arrows in $\C$, of certain reflexive
graphs in $\Fib(\C)$: given a reflexive graph of the following
form
\begin{equation*}\label{SpecialRG}
\vcenter{\xymatrix@=30pt{ R \ar@<-1ex>[r] \ar@{-{>>}}[d]_-{f''} \ar@<1ex>[r] & A' \ar[l] \ar@{-{>>}}[d]^-{f'} \ar@{->>}[r]^-{g} & A \ar@{-{>>}}[d]^-f\\
\Eq(h)  \ar@<-1ex>[r] \ar@<1ex>[r]
 & B'  \ar[l] \ar@{->>}[r]_-{h} & B }}
\end{equation*}
and its coequaliser, if $f'$ and $f''$ are normal extensions, then
also $f$ is a normal extension.
\end{enumerate}

We thus obtain
\begin{theorem}\label{Theorem Reflectiveness (G7minus)}
Let \(\Gamma=(\C,\X,I,H,\eta, \epsilon, \E, \F)\) be an admissible
Galois structure such that the conditions {\rm (G6)} and {\rm
(G7${}^{-}$)} hold. For any object $B$ in~$\C $, $\Norm(B) $ is a
reflective subcategory of~${(\E\downarrow B)}$. As a consequence,
normal extensions are reflective amongst fibrations.\noproof
\end{theorem}


\section{$\s $-protomodular categories} \label{Section S-protomodular categories}

Our criterion for the reflectiveness of normal extensions (Theorem
\ref{Theorem Reflectiveness (G7minus)}) can be applied to a
general algebraic situation, in which the category $\C $ is an
\emph{$\s $-protomodular category}. The aim of this section is to
recall the definition of an $\s $-protomodular category, as well
as the results we need in order to show that this reflectiveness
criterion is applicable.

The notion of $\s $-protomodular category was introduced for a
pointed context in~\cite{SchreierBook}, and further developed in
\cite{S-proto}. An extension to the non-pointed case was then
considered in~\cite{Bourn2014}.

Let $\C $ be a finitely complete category. We denote by $\Pt(\C) $
the \defn{category of points} in $\C $, whose objects $(f,s) $ are
the split epimorphisms $f\colon {A\to B}$ with a chosen section $s
\colon {B\to A}$ as in
\[
\xymatrix{ A \ar@<-.5ex>[r]_f & B \ar@<-.5ex>[l]_s}\qquad\qquad
f\comp s=1_{B}
\]
and whose morphisms are pairs of morphisms which form commutative
squares with both the split epimorphisms and their sections. Since
split epimorphisms are stable under pullbacks, the functor $\cod
\colon \Pt(\C) \to \C $, which associates with every split
epimorphism its codomain, is a fibration, usually called the
\defn{fibration of points}. Let $\s $ be a class of points in $\C
$ which is stable under pullbacks. If we look at it as a full
subcategory $\SPt(\C) $ of $\Pt(\C) $, it gives rise to a
subfibration $\s $-$\cod $ of the fibration of points. A point
$(f\colon{A\to B},s\colon B\to A) $ in a pointed category $\C $ is said to be a
\defn{strong point} if the pair $(k, s) $, where $k $ is a kernel
of $f $, is jointly strongly epimorphic. Strong points were
considered in~\cite{MartinsMontoliSobral2}, under the name of
\emph{regular points}, and independently in \cite{Bourn-monad},
under the name of \emph{strongly split epimorphisms}.

\begin{definition}[\cite{SchreierBook}, Definition $8.1.1$] \label{S-protomodular category}
Let $\C $ be a pointed finitely complete category, and $\s $ a
pullback-stable class of points. We say that $\C $ is
\defn{$\s $-proto\-modular} when:
\begin{itemize}
\item[(1)] every point in $\SPt(\C) $ is a strong point;
\item[(2)]$\SPt(\C) $ is closed under finite limits in $\Pt(\C) $.
\end{itemize}
\end{definition}

\begin{remark}\label{non-pointed S-protomodular}
As mentioned in~\cite{Bourn2014}, in a pointed finitely complete
category $\C $ a point $(f,s) $ is strong if and only if, for any
pullback as in the diagram
\[
\xymatrix@C=30pt{ P \ophalfsplitpullback \ar[r]^-{\pi_2} \ar@<0.5ex>[d]^(.6){\pi_1} & A \ar@<0.5ex>[d]^{f}\\
 C \ar[r]_-{g} \ar@<0.5ex>[u] & B, \ar@<0.5ex>[u]^s }
 \]
the pair $(\pi_2, s) $ is jointly strongly epimorphic. Thanks to
this fact, the definition of $\s $-protomodular category can be
extended to the non-pointed case, by simply replacing the notion
of strong point by the property above
(see~\cite[Definition~4.3]{Bourn2014}).
\end{remark}

The name \emph{$\s $-protomodular} comes from the fact that a
pointed finitely complete category $\C $ is protomodular if and
only if every point in $\C $ is a strong point~\cite{Bourn1991}.
Hence the notion above is a version of the concept of protomodular
category, relative with respect to the class $\s $.

\begin{example}\label{first examples}
As observed in~\cite{SchreierBook}, the categories $\Mon $ of
monoids and $\SRng $ of semirings are $\s $-protomodular with
respect to the class $\s $ of \emph{Schreier split epimorphisms}
\cite{BM-FMS} (see below). Later, in~\cite{MartinsMontoliSH}, it
was proved that every \defn{J\'{o}nsson-Tarski variety}, which is
a variety whose corresponding theory contains a unique
constant~$0$ and a binary operation $+$ which satisfy the
equations $0+x = x+0 = x $ for all $x $, is $\s $-protomodular
with respect to the class of Schreier split epimorphisms. Let us
now recall the definition of such split epimorphisms.
\end{example}

\begin{definition}[\cite{BM-FMS, MartinsMontoliSH}] \label{Def: Schreier split epi}
A split epimorphism $f\colon{A\to B}$ with given splitting $s
\colon {B\to A}$ in a J\'{o}nsson-Tarski variety is a
\defn{Schreier split epimorphism} when, for every $a \in A$, there
exists a unique $\alpha $ in the kernel $N$ of $f $ such that $a =
\alpha + sf(a) $.
\end{definition}

In Section~\ref{Section additive core} we give an example of an
$\s $-protomodular category of a different nature.

Let $\C $ be an $\s $-protomodular category. We recall
from~\cite{S-proto} that an \defn{$\s $-reflexive graph} (or
\defn{$\s $-reflexive relation})
$$
\xymatrix{Q \ar@<1ex>[r]^-{d} \ar@<-1ex>[r]_-{c} & A \ar[l]|-{e}}
$$
is a reflexive graph (respectively, a reflexive relation) such
that the point $(d,e) $ belongs to $\s $. A morphism $f\colon A\to
B$ is called an
\defn{$\s $-special morphism} when its kernel pair $\Eq(f) $ is an
$\s $-reflexive relation. An object $X$ is called an
\defn{$\s $-special object} when the indiscrete relation on $X$ is an
$\s $-reflexive relation. This means that the point
$(p_{1}\colon{X\times X\to X},\langle 1_{X}, 1_{X} \rangle\colon
{X\to X\times X}) $, where~$p_1$ is the first projection, belongs
to $\s $. The following result was proved, in the pointed case,
in~\cite{S-proto}, and then extended with the same proof to the non-pointed case in~\cite{Bourn2014}.

\begin{proposition}[\cite{S-proto}, Proposition $6.2$] \label{subcategory of S-special objects is protomodular}
Let $\C $ be an $\s $-protomodular category. Any split epimorphism
between $\s $-special objects is in $\s $ and, consequently, is an
$\s $-special morphism. The full subcategory $\s \C $ of $\s
$-special objects is protomodular.\noproof
\end{proposition}

The protomodular subcategory $\s \C $ is called the
\defn{protomodular core} of $\C $ relatively to the class $\s $.
Observe that, since $\SPt(\C) $ is closed under finite limits in~$\Pt(\C) $, the subcategory $\s \C $ is closed under finite limits in $\C $.

When $\C $ is the category of monoids, and $\s $ is the class of
Schreier split epimorphisms, the protomodular core is the category
of groups. Similarly, the protomodular core of the category of
semirings is the category of rings.


\section{An application to $\s $-protomodular categories}\label{Section S-application}
In this section we are going to consider a Galois structure
$\Gamma $ as in Definition~\ref{Def: Galois structure basic},
where $\C $ is a finitely complete Barr-exact category with
coequalisers of reflexive graphs, $\X$ is a full reflective
subcategory of~$\C $, $I$ is the reflector, $H$ is the inclusion
and $\E $ and $\F $ are the classes of regular epimorphisms. We
assume that
\begin{enumerate}
\item $\X$ is also Barr-exact; \item $H$ preserves regular
epimorphisms, so that $\Gamma $ is indeed a Galois structure;
\item $\Gamma $ is admissible; \item writing $\s $ for the class
of split epimorphic trivial extensions, the category $\C $ is $\s
$-protomodular.
\end{enumerate}
The functor $H$ being the inclusion functor, we omit it from
writing to simplify notation. Note that, $\s $ being
the class of split epimorphic trivial extensions, $\X$ is
contained in the protomodular core $\s \C$ given by $\s $-special objects: if $X \in \X $, then the
first projection $p_{1}\colon{X\times X\to X}$ is a trivial extension (because it is a morphism in $\X$). If $\C $ is pointed, then
$\X$ is precisely the protomodular core  $\s \C $. Indeed, if
$p_{1}\colon{X\times X\to X}$ is a trivial extension, then it is
the pullback of a morphism in $\X $. Hence its kernel, which is
$X$, belongs to $\X $. In any case, $\X $ is a full subcategory
of the protomodular core  $\s \C $, and being
closed under finite limits in it (since it is closed under finite
limits in $\C $), it is a protomodular category thanks to
Proposition~\ref{subcategory of S-special objects is
protomodular}, thus a Mal'tsev category (Proposition~17
in~\cite{B0}). Since~$\X$ is a Barr-exact Mal'tsev category, then
any reflexive relation is necessary the kernel pair of its
coequaliser.

Applying Theorem~\ref{Theorem Reflectiveness (G7minus)}, we shall
prove that in this setting, the normal extensions are reflective
amongst the fibrations. Since condition (G6) is fulfilled by
assumption, we only have to prove that condition (G7${}^{-}$)
holds.

In a regular category, a commutative square of regular
epimorphisms
$$
 \xymatrix{ A' \ar@{>>}[r]^-g \ar@{>>}[d]_-{f'} & A \ar@{>>}[d]^-f \\
 B' \ar@{>>}[r]_-h & B}
$$
is called a \defn{regular pushout}~\cite{Bourn2003} when the
comparison morphism to the pullback $\langle f',g\rangle \colon
{A' \to B'\times_B A}$ is a regular epimorphism.

\begin{lemma}\label{lemma regular pushouts}
In a regular category, pulling back along a morphism of regular
epimorphisms preserves regular pushout squares.
\end{lemma}
\begin{proof}
A square of regular epimorphisms as above is a regular pushout if
and only if it decomposes as a composite of two squares of regular
epimorphisms
\[
\xymatrix{A' \ar@{->>}[r] \ar@{->>}[d] & B'\times_{B}A \pullback \ar@{->>}[r] \ar@{->>}[d] & A \ar@{->>}[d]\\
B' \ar@{=}[r] & B' \ar@{->>}[r]_-{h} & B,}
\]
where the square on the right is a pullback. Given a regular
epimorphism $r\colon{C'\to C}$ and a morphism $(f',f)\colon {r\to
h}$, pulling back the given regular pushout square along it yields
a regular pushout square over $r $.
\end{proof}

\begin{lemma}
\label{3x3} Any commutative solid diagram
$$
\xymatrix{
 \Eq(f) \ar@<-1ex>@{.>}[d]_-{f_1} \ar@<1ex>@{.>}[d]^-{f_2} \ar@{.>}[r]^-{\overline h} &
 \Eq(g) \ar@<-1ex>@{.>}[d]_-{g_1} \ar@<1ex>@{.>}[d]^-{g_2} \\
 A \ar@{>>}[d]_-f \ar@{>>}[r]^-h \ar@{.>}[u] & C \ar@{>>}[d]^-g \ar@{.>}[u] \\
 B \ar@{>>}[r]_-k & D, \pushout}
$$
where the bottom square $gh=kf $ is a pushout of regular
epimorphisms and $f $ is a trivial extension is a \emph{regular
pushout}. Consequently, the comparison morphism $\overline{h}$ is
also a regular epimorphism.
\end{lemma}
\begin{proof}
By Proposition~5.4 and Theorem~5.5
in~\cite{Carboni-Kelly-Pedicchio} it suffices to prove that
$\R{h}$ and $\R{f}$ permute to show that the bottom square is a
regular pushout. The equality $\R{h}\R{f}=\R{f}\R{h}$ can be
proved with an argument which is completely analogous to the one
used in the proof of Theorem $3.9$ in~\cite{Bourn-quandles}.
\end{proof}

We recall that kernel pairs in $\Pt(\C) $ are computed objectwise:
if $(g,h) $ is a morphism of points, then $\Eq((g,h))=(\Eq(g),
\Eq(h)) $. Moreover, when $\C $ is regular, a morphism $(g,h) $ in
$\Pt(\C) $ is a regular epimorphism if and only if both $g $ and
$h $ are regular epimorphisms in $\C $.

\begin{lemma}\label{lemma reflexive graphs}
The functor $\triv|_{\Pt(\C)}\colon\Pt(\C)\to \Pt(\C) $ preserves
coequalisers of (effective) equivalence relations.
\end{lemma}
\begin{proof}
Consider the coequaliser diagram
$$
\vcenter{\xymatrix@=30pt{ \Eq(g) \ar@<-1ex>[r]_-{g_{2}}
\ar@<.5ex>[d]^-{f''} \ar@<1ex>[r]^-{g_{1}} & A' \ar[l]
\ar@<.5ex>[d]^-{f'} \ar@{->>}[r]^-{g} & A
\ar@<.5ex>[d]^-f\\
\Eq(h) \ar@<.5ex>[u]^-{s''} \ar@<-1ex>[r]_-{h_{2}}
\ar@<1ex>[r]^-{h_{1}}
 & B' \ar@<.5ex>[u]^-{s'} \ar[l] \ar@{->>}[r]_-{h} & B \ar@<.5ex>[u]^-{s}}}
$$
in $\Pt(\C) $. Since $I$ preserves all coequalisers, we obtain a
reflexive graph in~$\Pt(\X) $ with its coequaliser
\begin{equation*}
\vcenter{\xymatrix@=30pt{ I(\Eq(g)) \ar@<-1ex>[r]_-{I(g_2)}
\ar@<.5ex>[d] \ar@<1ex>[r]^-{I(g_1)} & I(A') \ar[l] \ar@<.5ex>[d]
\ar@{->>}[r]^-{I(g)} & I(A)
\ar@<.5ex>[d]\\
I(\Eq(h)) \ar@<.5ex>[u] \ar@<-1ex>[r]_-{I(h_2)}
\ar@<1ex>[r]^-{I(h_1)}
 & I(B') \ar@<.5ex>[u] \ar[l] \ar@{->>}[r]_-{I(h)} & I(B). \ar@<.5ex>[u]}}
\end{equation*}
The inclusion $\X\to \C $ preserves regular epimorphisms (by
assumption) and kernel pairs, so this diagram is still a reflexive
graph with its coequaliser when considered in the category
$\Pt(\C) $. Indeed, if we take the (regular epimorphism,
monomorphism) factorisation of $\langle I(g_1),I(g_2)\rangle\colon
I(\Eq(g))\to I(A')\times I(A') $ in~$\X$, we get a reflexive
relation, say $\langle e_1,e_2 \rangle\colon E\to I(A')\times
I(A') $, and the coequaliser of $(e_1,e_2) $ is still $I(g) $.
Since $\X$ is a Barr-exact Mal'tsev category, $E $ is necessarily
the kernel pair of its coequaliser $I(g) $, as mentioned above.
Thus, the comparison ${I(\Eq(g))\to \Eq(I(g))}$ is a regular
epimorphism, and similarly for ${I(\Eq(h))\to \Eq(I(h))}$.

Now we pull back along $\eta_{B}$, $\eta_{B'}$, $\langle I(h_1),
I(h_2) \rangle \comp \eta_{\Eq(h)}$ and $\eta_{\Eq(h)}$ to obtain
the diagram
\begin{equation*}
\resizebox{\textwidth}{!}{ \xymatrix@!0@C=5em@R=4em{
\Eq(g)_{\triv} \cubepullbackplus \ar@{.>>}[rr] \ar[rd]
\ar@<.5ex>[dd]^(.75){\triv(f'')} && P \cubepullback \ar[rd]
\ar@<.5ex>[dd]|(.5){\hole} \ar@<1ex>[rr] \ar@<-1ex>[rr] &&
A'_{\triv} \cubepullbackminus \ar[rd] \ar@{.>>}[rr]
\ar@<.5ex>[dd]^(.75){\triv(f')}|(.47){\hole}|(.5){\hole}|(.53){\hole}
\ar[ll] && A_{\triv}\cubepullbackminus
\ar@<.5ex>[dd]^(.75){\triv(f)}|(.5){\hole} \ar[rd]
\\
& I(\Eq(g)) \ar@{>>}[rr] \ar@<.5ex>[dd]^(.25){I(f'')} && \Eq(I(g))
\ar@<.5ex>[dd] \ar@<1ex>[rr] \ar@<-1ex>[rr] && I(A')
\ar@{>>}[rr]^(.25){I(g)}
\ar@<.5ex>[dd]^(.25){I(f')} \ar[ll] & & I(A) \ar@<.5ex>[dd]^-{I(f)} \\
\Eq(h) \ar[rd]_-{\eta_{\Eq(h)}} \ar@<.5ex>[uu]
\ar@{=}[rr]|(.48){\hole}|(.52){\hole} && \Eq(h) \ar[rd]
\ar@<.5ex>[uu]|(.5){\hole} \ar@<1ex>[rr]|(.48){\hole}|(.52){\hole}
\ar@<-1ex>[rr]|(.48){\hole}|(.52){\hole} && B'
\ar@{>>}[rr]|(.48){\hole}|(.52){\hole}^(.25)h \ar[rd]_-{\eta_B'}
\ar@<.5ex>[uu]|(.47){\hole}|(.5){\hole}|(.53){\hole} \ar[ll]|(.48){\hole}|(.52){\hole} && B \ar[rd]_-{\eta_B} \ar@<.5ex>[uu]|(.5){\hole} \\
& I(\Eq(h)) \ar@<.5ex>[uu] \ar@{>>}[rr]
 && \Eq(I(h)) \ar@<.5ex>[uu] \ar@<1ex>[rr] \ar@<-1ex>[rr] && I(B') \ar@{>>}[rr]_-{I(h)} \ar[ll] \ar@<.5ex>[uu] && I(B); \ar@<.5ex>[uu]} }
\end{equation*}
we write $P=\Eq(h)\times_{\Eq(I(h))} \Eq(I(g)) $ to simplify
notation. Since the front left and right faces are regular
pushouts (Proposition 3.2 in~\cite{Bourn2003}), the dotted arrows
are regular epimorphisms by Lemma~\ref{lemma regular pushouts}.
Moreover, pullbacks preserve kernel pairs, so that $P$ must be the
kernel pair of the regular epimorphism $A'_{\triv} \to A_{\triv}$.
Consequently, $\triv{(f)}$, being the coequaliser of its kernel
pair, is also the coequaliser of the reflexive graph
$\triv{(f'')\rightrightarrows\triv(f')}$.
\end{proof}

\begin{proposition}\label{Proposition kernel pair}
Consider a reflexive graph and its coequaliser in $\Pt(\C) $
$$
\vcenter{\xymatrix@=30pt{ R \ar@<-1ex>[r] \ar@<.5ex>[d]^-{f''} \ar@<1ex>[r] & A' \ar[l] \ar@<.5ex>[d]^-{f'} \ar@{->>}[r]^-{g} & A \ar@<.5ex>[d]^-f\\
S \ar@<.5ex>[u]^-{s''} \ar@<-1ex>[r] \ar@<1ex>[r]
 & B' \ar@<.5ex>[u]^-{s'} \ar[l] \ar@{->>}[r]_-{h} & B, \ar@<.5ex>[u]^-{s}}}
$$
where $f''$ and $f'$ are split epimorphic trivial extensions. Then
$f $ is also a split epimorphic trivial extension.
\end{proposition}
\begin{proof}
We first consider the situation where $R=\Eq(g) $ and $S=\Eq(h) $
are the kernel pairs of $g $ and $h $, respectively. By
assumption, $f $ is the coequaliser of its kernel pair
$$
\vcenter{\xymatrix@=30pt{ \Eq(g) \ar@<-1ex>[r]
\ar@<.5ex>[d]^-{f''} \ar@<1ex>[r] & A' \ar[l] \ar@<.5ex>[d]^-{f'}
\ar@{->>}[r]^-{g} & A \ar@<.5ex>[d]^-f
\ar[r]^{\cong} & A_{\triv} \ar[ld]^-{\triv(f)} \\
\Eq(h) \ar@<.5ex>[u]^-{s''} \ar@<-1ex>[r] \ar@<1ex>[r]
 & B' \ar@<.5ex>[u]^-{s'} \ar[l] \ar@{->>}[r]_-{h} & B. \ar@<.5ex>[u]^-{s}}}
$$
But, applying Lemma~\ref{lemma reflexive graphs}, we conclude that
$\triv(f) $ is also its coequaliser, since $\triv(f')=f'$ and
$\triv(f'')=f''$. Thus $\triv(f) $ and $f $ are isomorphic, which
proves that $f $ is a trivial extension.

Now we prove that the above assumption can be made without any
loss of generality. Consider the diagram
$$
\resizebox{\textwidth}{!}{ \xymatrix@!0@C=5em@R=4em{ R
\cubepullback \ar[rr]^-{\rho} \ar@<.5ex>[dd]^-{f''}
\ar[dr]^-{\eta_R} && P \cubepullback \ar@<-1ex>[rr]_-{p_2} \ar[rd]
\ar@<.5ex>[dd]|(.5){\hole} \ar@<1ex>[rr]^-{p_1} && A'
\cubepullback \ar[rd]^-{\eta_{A'}} \ar[ll]
\ar@<.5ex>[dd]|(.47){\hole}|(.5){\hole}|(.53){\hole}^(.75){f'}
\ar@{->>}[rr]^-{g} && A \ar[rd]^-{\eta_A} \ar@<.5ex>[dd]^(.25)f|(.5){\hole}\\
& I(R) \ar@{>>}[rr]^(.25){\gamma} \ar@<.5ex>[dd]^(.25){I(f'')} &&
\Eq(I(g)) \ar@<-1ex>[rr] \ar@<.5ex>[dd] \ar@<1ex>[rr] && I(A')
\ar[ll]
\ar@<.5ex>[dd]^(.25){I(f')} \ar@{->>}[rr]^(.25){I(g)} && I(A) \ar@<.5ex>[dd]^-{I(f)}\\
S \ar[dr]_-{\eta_S} \ar@<.5ex>[uu]^-{s''}
\ar[rr]|(.48){\hole}|(.52){\hole} && \Eq(h) \ar[rd]
\ar@<.5ex>[uu]|(.5){\hole}
\ar@<-1ex>[rr]|(.48){\hole}|(.52){\hole}
\ar@<1ex>[rr]|(.48){\hole}|(.52){\hole}
 && B' \ar[rd]_-{\eta_{B'}} \ar@<.5ex>[uu]^(.25){s'}|(.47){\hole}|(.5){\hole}|(.53){\hole} \ar[ll]|(.48){\hole}|(.52){\hole}
 \ar@{->>}[rr]^(.25){h}|(.48){\hole}|(.52){\hole} && B \ar[rd]_-{\eta_B} \ar@<.5ex>[uu]|(.5){\hole}^(.75){s}\\
& I(S) \ar@{>>}[rr] \ar@<.5ex>[uu] && \Eq(I(h)) \ar@<.5ex>[uu]
\ar@<-1ex>[rr] \ar@<1ex>[rr]
 && I(B') \ar@<.5ex>[uu] \ar[ll] \ar@{->>}[rr]_-{I(h)} && I(B), \ar@<.5ex>[uu]} }
$$
where $P=\Eq(h)\times_{\Eq(I(h))} \Eq(I(g)) $. We shall prove that
$P$ is precisely the kernel pair of $g $, so that the induced
split epimorphism ${\Eq(g)\to \Eq(h)}$ is a trivial extension,
being a pullback of a fibration in $\X$
(Proposition~\ref{Proposition Trivial implies normal}).

For $P$ to be the kernel pair of $g $, we just need to show that
$g\comp p_1=g\comp p_2$, since the rest of the proof is
straightforward. As in the previous proof, the comparison
morphisms ${I(R)\to \Eq(I(g))}$ and ${I(S)\to \Eq(I(h))}$ are
regular epimorphisms, so that the front left square of the diagram
above is a regular pushout (Proposition~3.2 in~\cite{Bourn2003}).
Consequently, the comparison morphism
\[
\langle I(f''), \gamma \rangle\colon I(R) \to
I(S)\times_{\Eq(I(h))} \Eq(I(g))
\]
is a regular epimorphism and so is the comparison morphism
$\langle f'', \rho \rangle $ in
$$
\xymatrix@C=15pt@R=10pt{
 R \ar[rrr]^-{\rho} \ar@{.>>}[rd]|(.35){\langle f'', \rho \rangle} \ar@<.5ex>[ddd]^-{f''} & & & P \ar@<.5ex>[ddd] \\
 & S\times_{\Eq(h)} P \ar@<.5ex>[ddl] \ar[rru]^(.4){p_P} \\ \\
 S \ar[rrr] \ar@<.5ex>[uuu] \ar@<.5ex>[uur] & & & \Eq(h), \ar@<.5ex>[uuu]^-t
}
$$
as a pullback of $\langle I(f''), \gamma \rangle $. The split
epimorphism $\Eq(I(g)) \leftrightarrows \Eq(I(h)) $ belongs to $\s
$ by Proposition~\ref{subcategory of S-special objects is
protomodular}, and so does the split epimorphism $P
\leftrightarrows \Eq(h) $ by the assumption of stability under
pullbacks. Since $\C $ is an $\s $-protomodular category, the pair
$(p_P,t) $ is jointly strongly epimorphic, thus jointly epimorphic
(Remark~\ref{non-pointed S-protomodular}). Then, the pair $(\rho,
t) $ is jointly epimorphic, so we get $g\comp p_1=g\comp p_2$.
This finishes the proof.
\end{proof}

We have the following partial converse of
Proposition~\ref{Proposition kernel pair}.

\begin{proposition}\label{Real Proposition kernel pair}
Consider a morphism of points and its kernel pair in $\Pt(\C) $
$$
\vcenter{\xymatrix@=30pt{ \Eq(g) \ar@<-1ex>[r]_-{g_{2}}
\ar@<.5ex>[d]^-{f''} \ar@<1ex>[r]^-{g_{1}} & A' \ar[l]
\ar@<.5ex>[d]^-{f'} \ar[r]^-{g} & A
\ar@<.5ex>[d]^-f\\
\Eq(h) \ar@<.5ex>[u]^-{s''} \ar@<-1ex>[r]_-{h_{2}}
\ar@<1ex>[r]^-{h_{1}}
 & B' \ar@<.5ex>[u]^-{s'} \ar[l] \ar[r]_-{h} & B \ar@<.5ex>[u]^-{s}}}
$$
where $f $ and $f'$ are split epimorphic trivial extensions. Then
$f''$ is also a split epimorphic trivial extension.
\end{proposition}
\begin{proof}
This follows from the finite limit closure in the definition of
$\s $-protomodul\-arity (Definition~\ref{S-protomodular
category}).
\end{proof}

Since the class $\s $ we are considering is the class of split
epimorphic trivial extensions, then the $\s $-special regular
epimorphisms are precisely the normal extensions with respect to
the Galois structure $\Gamma $ (Definition~\ref{trivial, central,
normal}). We are now ready to prove that condition (G7${}^{-}$)
holds.

\begin{proposition}\label{Proposition coequaliser of reflexive graphs} The category of $\s$-special regular epimorphisms
is closed in $\Arr(\C)$ under coequalisers of reflexive graphs,
when they are of the type considered in condition {\rm
(G7${}^{-}$)}.
\end{proposition}
\begin{proof}
Consider a reflexive graph of regular epimorphisms and its
coequaliser in $\C$ as in the solid part of the diagram in
Figure~\ref{Figure (G7) monoids}. Assume that $S$ is an
equivalence relation, so that $S = \Eq(h)$. We prove that, if
$f''$ and $f'$ are $\s$-special regular epimorphisms, then also
$f$ is an $\s$-special regular epimorphism.

\begin{figure}[h]
\(\xymatrix@=30pt{ \Eq(f'') \ar@{.>}@<-1ex>[d] \ar@{.>}@<1ex>[d]
\ar@{.>}@<-1ex>[r] \ar@{.>}@<1ex>[r] & R \ar@{.>}[l] \ar@<-1ex>[d]
\ar@{->>}[r]^-{f''} \ar@<1ex>[d] &
 \Eq(h) \ar@<-1ex>[d] \ar@<1ex>[d] \\
\R{f'} \ar@{.>}[u] \ar@{.>}[d]_{\overline{g}} \ar@{.>}@<-1ex>[r]
\ar@{.>}@<1ex>[r] & A' \ar@{.>}[l] \ar[u] \ar@{->>}[r]^-{f'}
\ar@{->>}[d]_-{g} & B' \ar[u]
\ar@{->>}[d]^-{h} \\
\R{f} \ar@{.>}@<-1ex>[r] \ar@{.>}@<1ex>[r] & A \ar@{.>}[l]
\ar@{->>}[r]_-f & B}\) \caption{Closedness of $\s $-special
regular epimorphisms under coequalisers of certain reflexive
graphs}\label{Figure (G7) monoids}
\end{figure}

Taking kernel pairs to the left, we want to use
Proposition~\ref{Proposition kernel pair} together with the fact
that $\s$-special regular epimorphisms are precisely normal
extensions to show that the kernel pair projections of $f$ are
trivial extensions. For this argument to be valid, we need to show
that: (1) $\overline{g}$ is a regular epimorphism; and (2) it is
the coequaliser of the pair of vertical arrows
$\R{f''}\rightrightarrows \R{f'}$.

We may deduce (1) that $\overline{g}$ is a regular epimorphism
from the fact that the coequaliser of $\R{f''}\rightrightarrows
\R{f'}$
$$
\xymatrix@=30pt{
 \Eq(f'') \ar@<1ex>[r] \ar@<-1ex>[r] \ar@<1ex>[d] \ar@<-1ex>[d] & R \ar@<1ex>[d] \ar@<-1ex>[d] \ar[l]\\
 \Eq(f') \ar@<1ex>[r]^-{f'_1} \ar@<-1ex>[r]_-{f'_2} \ar@{>>}[d] \ar[u] & A' \ar@{>>}[d]^-g \ar[l] \ar[u] \\
 Q \ar@<1ex>[r]^-d \ar@<-1ex>[r]_-c & A \ar[l]}
$$
is an internal groupoid on $A$. Indeed, by
Proposition~\ref{Proposition kernel pair}, it is an $\s$-reflexive
graph since $d$ is a split epimorphic trivial extension. Thanks to
Proposition~7.5 in~\cite{S-proto} (and to its extension to the
non-pointed context, see Proposition $4.9$ in~\cite{Bourn2014}),
it suffices then to show that the kernel pairs $\Eq(d)$ and
$\Eq(c)$ centralise each other. The kernel pairs $\Eq(f'_1)$ and
$\Eq(f'_2)$ centralise each other, since $\Eq(f')$ is an
equivalence relation. By Lemma~\ref{3x3}, $\Eq(d)$ (resp.\
$\Eq(c)$) is the regular image of $\Eq(f'_1)$ (resp.\
$\Eq(f'_2)$), so that $\Eq(d)$ and $\Eq(c)$ centralise each other
too (Proposition 1.6.4 in~\cite{Borceux-Bourn}). Hence the regular
image of this internal groupoid is an equivalence relation, so a
kernel pair, with coequalizer $f$, which makes it isomorphic
to~$\Eq(f)$.

Observe that, in the proof of (1), we do not need $S$
to be an equivalence relation.

For the proof of (2), write $f'''\colon {\Eq(g) \to \Eq(h)}$ for
the kernel pair of $(g,h)$. Taking kernel pairs to the left, we
obtain the kernel pair projections $\R{f'''}\rightrightarrows
\Eq(g)$. Note that $\R{f'''}$ is actually the kernel pair of
$\overline{g}$ by interchange of limits. We claim that the
comparison $R\to \Eq(g)$ is a regular epimorphism. Hence, by
pullback, so is the comparison $\R{f''}\to \R{f'''}$, which
finishes the proof of (2).

We are left with proving our claim that $R\to \Eq(g)$ is a regular
epimorphism. We do so by showing that there is a quotient $R'$ of
$R $ which is a groupoid, so that the ``image'' of the reflexive
graph $R $ is an (effective) equivalence relation
(namely~$\Eq(g)$). The
groupoid $R'$ is obtained as a pullback of groupoids like in the
diagram
$$
\xymatrix@!0@C=5em@R=4em{ R \cubepullback \ar[rr]^-{\rho}
\ar@{>>}[dd]^-{f''} \ar[dr]^-{\eta_R} && R' \cubepullback
\ar@<-1ex>[rr]_-{p_2} \ar[rd] \ar@{>>}[dd]|(.5){\hole}
\ar@<1ex>[rr]^-{p_1} && A' \cubepullback \ar[rd]^-{\eta_{A'}}
\ar[ll]
\ar@{>>}[dd]|(.47){\hole}|(.5){\hole}|(.53){\hole}^(.75){f'}
\\
& I(R) \ar@{>>}[rr] \ar@{>>}[dd]^(.25){I(f'')} && \grd(I(R))
\ar@<-1ex>[rr] \ar@{>>}[dd] \ar@<1ex>[rr] && I(A') \ar[ll]
\ar@{>>}[dd]^(.25){I(f')} \\
\Eq(h) \ar[dr]_-{\eta_S} \ar@{=}[rr]|(.5){\hole} && \Eq(h) \ar[rd]
\ar@<-1ex>[rr]|(.5){\hole} \ar@<1ex>[rr]|(.5){\hole}
 && B' \ar[rd]_-{\eta_{B'}}  \ar[ll]|(.5){\hole}
 \\
& I(S) \ar@{>>}[rr] && \grd(I(S)) \ar@<-1ex>[rr] \ar@<1ex>[rr]
 && I(B')  \ar[ll] }
$$
where $\grd(I(R))$ and $\grd(I(S))$ are the groupoids associated
with the reflexive graphs $I(R)$ and $I(S)$, respectively. Since
$\X$ is a Barr-exact Mal'tsev category, the reflection of
reflexive graphs to groupoids is Birkhoff (Corollary~3.15
in~\cite{P0} combined with Theorem~3.1 in~\cite{Gran:Internal}),
so that (keeping Theorem~5.7 in~\cite{Carboni-Kelly-Pedicchio} in
mind) the front left square is a regular pushout. The morphism
$\rho$ is now a regular epimorphism by Lemma~\ref{lemma regular
pushouts}.
\end{proof}

\begin{corollary}\label{Proposition coequaliser of reflexive graphs Restricted} The category of $\s$-special regular epimorphisms in $\C$
is closed in $\Arr(\C)$ under coequalisers of equivalence
relations.\noproof
\end{corollary}

Theorem~\ref{Theorem Reflectiveness (G7minus)} now implies the
main result of this section.

\begin{theorem}\label{Reflectiveness general monoids}
$\s $-special regular epimorphisms are reflective amongst regular
epimorphisms.\noproof
\end{theorem}

We conclude this section by observing that the criterion for
reflectiveness of normal extensions given by Theorem $4.2$
in~\cite{JK:Reflectiveness} cannot be applied to obtain the
theorem above in our general framework, since we are not supposing
that the category $\C $ admits the colimits that are needed to
apply that theorem.


\section{Examples}
\label{Examples} In this section we describe some concrete
examples of the general framework developed in the previous one.

\subsection{Monoids and groups}
The first example we consider is the following: $\C = \Mon $ is
the category of monoids, and $\X = \Grp $ is the subcategory of
groups. The reflection $\gp \colon \Mon \to \Grp $ is given by the
\defn{Grothendieck group} (or \defn{group completion})~\cite{Maltsev:RingCompletion, Maltsev:GroupCompletion,
Maltsev:GroupCompletionII}: given a monoid $(M,\cdot,1) $, its
group completion $\gp(M) $ is defined by
\[
\gp(M)=\frac{\gp\free (M)}{\normal(M)},
\]
where $\gp \free (M) $ denotes the free group on $M$ and
$\normal(M) $ is the normal subgroup generated by elements of the
form $[m_{1}][m_{2}][m_{1}\cdot m_{2}]^{-1}$. By choosing the
classes of morphisms $\E $ and~$\F $ to be the surjections in
$\Mon $ and $\Grp $, respectively, we obtain a Galois structure
\begin{equation*}
\label{the adjunction}
 \Gamma_{\mon}=(\Mon,\Grp, \gp, \mon, \eta, \epsilon, \E, \F),
\end{equation*}
where $\mon $ is just the inclusion functor from $\Grp $ to $\Mon
$. This Galois structure was studied in~\cite{MRVdL}, where it was
shown to be admissible (Theorem~2.2 there). Moreover, trivial,
normal and central extensions were characterised for this Galois
structure. Let us briefly recall what they are.

\begin{definition}[\cite{SchreierBook}, Definition $2.1.1$] \label{mrvdl2:homogeneous split epi}
Let $f $ be a split epimorphism of monoids, with a chosen
splitting $\s $, and~$N$ its (canonical) kernel
\begin{equation*}
 \xymatrix{ N \ar@{{ |>}->}[r]_-{k} & A \ar@<-.5ex>[r]_-{f} & B. \ar@<-.5ex>[l]_-{s} }
\end{equation*}
The split epimorphism $(f,s) $ is said to be \defn{right
homogeneous} when, for every element $b \in B$, the function
$\mu_b \colon N \to f^{-1}(b) $ defined through multiplication on
the right by $s (b) $, so $\mu_b(n) = n \, s(b) $, is bijective.
Similarly, we can define a
\defn{left homogeneous} split epimorphism: the function $N \to
f^{-1}(b)\colon n\mapsto s(b)\, n $ is a bijection for all $b\in
B$. A~split epimorphism is said to be
\defn{homogeneous} when it is both right and left homogeneous.
\end{definition}

As observed in~\cite{SchreierBook}, Proposition $2.1.3$, a split
epimorphism is right homogeneous if and only if it is a Schreier
split epimorphism (Definition~\ref{Def: Schreier split epi}).

\begin{definition}[\cite{SchreierBook}, Definition $7.1.1$] \label{mrvdl2:special homogeneous surj}
Given a surjective homomorphism $g $ of monoids and its kernel
pair
\begin{equation*}
\label{mrvdl2:epi g}
 \xymatrix{ \R{g} \ar@<1ex>[r]^-{\pi_1} \ar@<-1ex>[r]_-{\pi_2} & A \ar[l]|-{\Delta} \ar@{-{>>}}[r]^-g & B, }
\end{equation*}
the morphism $g $ is called a \defn{special homogeneous
surjection} when $(\pi_1, \Delta) $ (or, equivalently,
$(\pi_2,\Delta) $) is a homogeneous split epimorphism.
\end{definition}

\begin{proposition}[\cite{MRVdL}, Proposition $4.2$]
\label{mrvdl2:Proposition: trivial split extensions} For a split
epimorphism $f $ of monoids, the following statements are
equivalent:
\begin{tfae}
 \item $f $ is a trivial extension;
 \item $f $ is a special homogeneous surjection.\noproof
\end{tfae}
\end{proposition}

\begin{theorem}[\cite{MRVdL}, Theorem 4.4]
\label{mrvdl2:Theorem: normal extensions} For a surjective
homomorphism~$g $ of monoids, the following statements are
equivalent:
\begin{tfae}
\item $g $ is a central extension; \item $g $ is a normal
extension; \item $g $ is a special homogeneous surjection.\noproof
\end{tfae}
\end{theorem}

Special homogeneous split epimorphisms are, in particular,
Schreier split epimorphisms, hence strong points
(\cite{SchreierBook}, Lemma 2.1.6). Moreover, they are stable
under pullbacks (\cite{SchreierBook}, Proposition $7.1.4$). So,
$\Mon $ is an $\s $-protomodular category with respect to the
class $\s $ of special homogeneous split epimorphisms, which are
precisely the split epimorphic trivial extensions of the Galois
structure $\Gamma_{\mon}$ we are considering. All the other
conditions we assumed in Section~\ref{Section S-application} are
clearly satisfied by $\Gamma_{\mon}$. As a consequence of Theorem
\ref{Reflectiveness general monoids}, we see that special
homogeneous surjections are reflective amongst surjective monoid
homomorphisms. We observe that this fact was already proved in
\cite{MRVdL2}, using Theorem $4.2$ in~\cite{JK:Reflectiveness}
(although, as we already mentioned, the same theorem cannot be
applied to the general framework of Section~\ref{Section
S-application}).

\subsection{Semirings and rings}
The second example we consider is of a similar nature. Now $\C =
\SRng $ is the category of semirings, and $\X = \Rng $ is the
reflective subcategory of rings. In order to describe the
reflection, we first restrict the group completion functor to
commutative monoids. This restriction has a simpler description
which we now recall. If $(M, + , 0) $ is a commutative monoid,
then its group completion $\gp(M) $ can be described as the
quotient ${M \times M}/{\sim}$, where $(m, n) \sim (p, q) $ when
there exists $k \in m $ such that
\[
m + q + k = n + p + k.
\]
Now let $(M, +, \cdot, 0) $ be a semiring; we can define a product
in $\gp(M) $ in the following way:
\[
[(m, n)] \cdot [(m', n')] = [(m \cdot m' + n \cdot n', m \cdot n'
+ n \cdot m')].
\]
It is easy to check that this definition does not depend on the
choice of the representative for the class in $\gp(M) $, and that
it turns $\gp(M) $ into a ring. Hence it gives the desired
reflection $\rng \colon \SRng \to \Rng $.

Via a simplified version of the arguments used in~\cite{MRVdL} for
the Galois structure between $\Mon $ and $\Grp $, it is not
difficult to see that the reflection of the adjunction between
$\SRng $ and $\Rng $ is admissible with respect to the classes of
surjective homomorphisms both in $\SRng $ and in $\Rng $. Hence we
get an admissible Galois structure. Once again, the split
epimorphic trivial extensions are precisely the special
homogeneous split epimorphisms, while the normal (=~central)
extensions are the special homogeneous surjections; the proofs
easily follow from those of Proposition~\ref{mrvdl2:Proposition:
trivial split extensions} and Theorem~\ref{mrvdl2:Theorem: normal
extensions}. Proposition $6.7.2$ in~\cite{SchreierBook} implies
that a split epimorphism $(f\colon{A\to B},s\colon {B\to A}) $ in
$\SRng $ is special homogeneous if and only if the kernel $N$ of
$f$ is a ring and $A$ is isomorphic to a semidirect product of $B$
and $N$. (Observe that every Schreier split epimorphism of
semirings is homogeneous, because the additive monoid structure is
commutative.) This implies, in particular, that $A$, as a monoid,
is the cartesian product of $B$ and $N$.

It is easy to see that all the conditions of Section~\ref{Section
S-application} are satisfied by this Galois structure. Hence
Theorem~\ref{Reflectiveness general monoids} implies, like for the
case of monoids and groups, that special homogeneous surjections
of semirings are reflective amongst surjective homomorphisms.
(Once again, we could also conclude this by applying Theorem~4.2
in \cite{JK:Reflectiveness}.)



\section{The additive core of a unital category} \label{Section additive core}

This section is devoted to the description of a general example of
the situation considered in Section~\ref{Section S-application}.
This example is of a rather different nature from the ones of the
previous section, so that Theorem $4.2$
of~\cite{JK:Reflectiveness} does not apply.

We start by recalling from~\cite{B0} that a pointed finitely
complete category $\C $ is \defn{unital} when, for every pair of
objects $(A, B) $ of $\C $, the morphisms $\langle 1_A, 0_{A,B}
\rangle $ and $\langle 0_{B,A}, 1_B \rangle $ in the product
diagram
\[
\xymatrix@C=40pt{ A \ar@<-.5ex>[r]_-{\langle 1_A, 0_{A,B} \rangle}
& A \times B \ar@<-.5ex>[l]_-{p_A} \ar@<.5ex>[r]^-{p_B} & B
\ar@<.5ex>[l]^-{\langle 0_{B,A}, 1_B \rangle} }
\]
are jointly strongly epimorphic.

Examples of unital categories are all J\'{o}nsson-Tarski varieties
(Example~\ref{first examples}). Actually, as shown
in~\cite[Theorem~1.2.15]{Borceux-Bourn}, a variety of universal
algebras is a unital category precisely when it is a
J\'{o}nsson-Tarski variety.

An object $X$ in a unital category $\C $ is called \defn{abelian}
when it carries an internal abelian group structure (which is
necessarily unique, as a consequence of Theorem~1.4.5
in~\cite{Borceux-Bourn}). The full subcategory of $\C $ determined
by the abelian objects is denoted $\Ab(\C) $ and called the
\defn{additive core} of $\C $. The category $\Ab(\C) $ is indeed
additive (by Corollary 1.10.13 in~\cite{Borceux-Bourn}), hence it
is protomodular (by Example~3.1.13 in~\cite{Borceux-Bourn}). If
$\C $ is a finitely cocomplete regular unital category, then
$\Ab(\C) $ is really a core, since it is a reflective subcategory
of $\C $ by Propositions 1.7.5 and 1.7.6 of~\cite{Borceux-Bourn}
$$
 \xymatrix{\C \ar@<1ex>[r]^-{\ab} \ar@{}[r]|-\bot & \Ab(\C); \ar@<1ex>[l]^-{\supset}}
$$
the unit is denoted by $\eta^{\ab}$. Since $\Ab(\C) $ is closed in
$\C $ under regular
epimorphisms~\cite[Proposition~1.6.11]{Borceux-Bourn}, this
adjunction gives a Galois structure with respect to the regular
epimorphisms in $\C $ and in $\Ab(\C) $; we denote it
by~$\Gamma_{\ab}$.

We now assume $\C $ to be a finitely cocomplete Barr-exact unital
category. We can then show that the Galois structure
$\Gamma_{\ab}$ satisfies all the conditions of
Section~\ref{Section S-application}. First of all, $\Ab(\C) $ is
also Barr-exact~\cite[Theorem~5.11]{Barr}. The additive core
$\Ab(\C) $ is then an abelian category, called the \defn{abelian
core} of~$\C $. Next, we shall prove that~$\C $ is an $\s
$-protomodular category, where $\s $ is the class of split
epimorphic trivial extensions. In fact, the split epimorphic
trivial extensions for the Galois structure $\Gamma_{\ab}$ have an
easy description: see Proposition~\ref{trivial split extensions
abelian core}.

\begin{lemma}\label{lemma products}
If $B$ is an object and $N$ an abelian object of $\C $ then
\[
\ab(N\times B)\cong N\times \ab(B).
\]
\end{lemma}
\begin{proof}
There is a comparison morphism
\[
\lambda\colon \ab(N\times B)\to N\times \ab(B)
\]
such that $\lambda\comp \eta_{N\times B}^{\ab} = 1_N\times
\eta_B^{\ab}$. We use the fact that binary products coincide with
binary coproducts in $\Ab(\C) $ and consider the morphism
$$
 \xi=\left\lgroup \eta_{N\times B}^{\ab}\comp \langle 1_N,0_{N,B} \rangle \;\; \ab(\langle 0_{B,N},1_B \rangle) \right\rgroup\colon N\oplus\ab(B) \to \ab(N\times B).
$$
Note that for the coproduct inclusions $i_N $ and $i_{\ab(B)}$ of
$N\oplus \ab(B) $, we have $i_N=\langle 1_N,0_{N,\ab(B)}\rangle $
and $i_{\ab(B)}=\langle 0_{\ab(B),N},1_{\ab(B)}\rangle $. Then
\begin{align*}
 \lambda \comp \xi \comp i_N & =
 \lambda \comp \eta_{N\times B}^{\ab}\comp \langle 1_N,0_{N,B}\rangle
 = (1_N\times \eta_B^{\ab})\comp \langle 1_N,0_{N,B}\rangle
 = \langle 1_N,0_{N,\ab(B)}\rangle
 = i_N
\end{align*}
and
\begin{align*}
 \lambda \comp \xi \comp i_{\ab(B)} \comp \eta_B^{\ab} & =
 \lambda \comp \ab(\langle 0_{B,N},1_B\rangle ) \comp \eta_B^{\ab}
 = \lambda\comp \eta_{N\times B}^{\ab}\comp \langle 0_{B,N},1_B\rangle \\
 & = (1_N\times \eta_B^{\ab}) \comp \langle 0_{B,N},1_B\rangle
 = \langle 0_{B,N},\eta_B^{\ab}\rangle \\
 & = \langle 0_{\ab(B),N},1_{\ab(B)}\rangle \comp \eta_B^{\ab}
 = i_{\ab(B)}\comp \eta_B^{\ab}.
\end{align*}
The universal property of the unit $\eta^{\ab}$ gives $\lambda
\comp \xi \comp i_{\ab(B)} = i_{\ab(B)}$, so that $\lambda \comp
\xi = 1_{N\oplus \ab(B)}$.

On the other hand, the equalities
\begin{align*}
 \xi \comp (1_N\times \eta_B^{\ab})\comp \langle 1_N,0_{N,B}\rangle & =
 \xi \comp \langle 1_N,0_{N,\ab(B)}\rangle
 = \xi \comp i_N
 = \eta_{N\times B}^{\ab}\comp \langle 1_N,0_{N,B}\rangle
\end{align*}
and
\begin{align*}
 \xi \comp (1_N\times \eta_B^{\ab})\comp \langle 0_{B,N},1_B\rangle & =
 \xi \comp \langle 0_{B,N},\eta_B^{\ab}\rangle
 = \xi \comp \langle 0_{\ab(B),N},1_{\ab(B)}\rangle \comp \eta_B^{\ab} \\
 & = \xi \comp i_{\ab(B)}\comp \eta_B^{\ab}
 = \ab(\langle 0_{B,N},1_{\ab(B)}\rangle )\comp \eta_B^{\ab} \\
 & = \eta_{N\times B}^{\ab}\comp \langle 0_{B,N},1_B\rangle
\end{align*}
show that $\xi \comp (1_N\times \eta_B^{\ab})=\eta_{N\times
B}^{\ab}$ since $\langle 1_N,0_{N,B}\rangle $ and $\langle
0_{B,N},1_B\rangle $ are jointly epimorphic, $\C $ being a unital
category. Finally, from
\[
 \xi \comp \lambda\comp \eta_{N\times B}^{\ab} =
 \xi \comp (1_N\times \eta_B^{\ab}) = \eta_{N\times B}^{\ab}
\]
we conclude that $\xi \comp \lambda=1_{\ab(N\times B)}$ by the
universal property of the unit $\eta^{\ab}$.
\end{proof}

\begin{proposition} \label{trivial split extensions abelian core}
Let $\C $ be a finitely cocomplete Barr-exact unital category.
A~split epimorphism $f\colon{A\to B}$ with splitting $s
\colon{B\to A}$ in $\C $ is a trivial extension with respect to
$\Gamma_{\ab}$ if and only if the following two conditions hold:
\begin{enumerate}
\item $(f,s) $ is isomorphic, as a point, to a product
\[
(p_{B}\colon{N\times B\to B},\langle0_{B,N},1_{B}\rangle\colon
{B\to N\times B});
\]
\item the kernel $N$ of $f $ is abelian.
\end{enumerate}
\end{proposition}

\begin{proof}
Let $(f,s) $ be a split epimorphic trivial extension. Then the
square
\[
 \xymatrix@C=30pt{A \halfsplitpullback \ar[r]^-{\eta_A^{\ab}} \ar@<0.5ex>[d]^-f & \ab(A) \ar@<0.5ex>[d]^-{\ab(f)}\\
 B \ar@<0.5ex>[u] \ar[r]_-{\eta_B^{\ab}} & \ab(B) \ar@<0.5ex>[u]}
 \]
is a pullback. So the kernel $N$ of $f $ is also the kernel of
$\ab(f) $, and is therefore abelian. Moreover, a split epimorphism
in $\Ab(\C) $ is a product projection and, consequently, $(f,s) $
is isomorphic to $(p_{B},\langle0_{B,N},1_{B}\rangle) $.

Conversely, we must show that any product projection
$(p_{B},\langle0,1_{B}\rangle) $, where $N$ is abelian, is a
trivial extension. To do so it suffices to show that
\[
\ab(N\times B)\cong N\times \ab(B),
\]
so that $\eta_{N\times B}\cong 1_N\times \eta_b $. This is
precisely Lemma~\ref{lemma products}.
\end{proof}

Thanks to this characterisation, we have that $\C $ is $\s
$-protomodular with respect to the class of split epimorphic
trivial extensions. This follows easily from the fact that a
pointed finitely complete category $\C $ is unital if and only if
it is $\s $-protomodular with respect to the class $\s $ of points
of the form $(p_{B},\langle0_{B,N},1_{B}\rangle) $---an
observation which is due to Sandra Mantovani.

The last condition of Section~\ref{Section S-application} we must
show to hold concerns the admissibility of the Galois structure
$\Gamma_{\ab}$.

\begin{theorem}
\label{theorem admissibility} Let $\C $ be a finitely cocomplete
Barr-exact unital category. The Galois structure $\Gamma_{\ab}$ is
admissible.
\end{theorem}
\begin{proof}
Combining Theorem 4.3 in~\cite{CHK} with both Definition 5.5.3 and
Proposition~5.5.5 in~\cite{Borceux-Janelidze}, we see that the
Galois structure $\Gamma_{\ab}$ is admissible if and only if every
pullback
\[
\xymatrix{X \pullback \ar@{>>}[d]_-{f} \ar[r]^-{a} & A \ar@{>>}[d]^-{g}\\
Y \ar[r]_-{b} & B}
\]
with $g $ a regular epimorphism in $\Ab(\C) $ is preserved by the
reflector $\ab $.

We first begin by supposing that $g $ is a split epimorphism,
hence a product projection. Then, being its pullback, so is the
split epimorphism $f $. Furthermore, the morphism $a $ in the
pullback is of the form $1_{N\times B}\colon{N\times Y\to N\times
B}$ with~$N $ abelian, and it follows from Lemma~\ref{lemma
products} that $\ab $ preserves such a pullback.

For the general case, we consider the diagram
$$
\xymatrix@=30pt{
 \Eq(f) \ar@<1ex>[d] \ar@<-1ex>[d] \ar[r]^-{\eta_{\Eq(f)}^{\ab}} & \ab(\Eq(f)) \ar@<1ex>[d] \ar@<-1ex>[d] \ar[r] & \Eq(g) \ar@<1ex>[d] \ar@<-1ex>[d]\\
 X \ar[u] \ar[r]^-{\eta_X^{\ab}} \ar@{>>}[d]_-f & \ab(X) \ar@{>>}[d]^-{\ab(f)} \ar[u] \ar[r] & A \ar[u] \ar@{>>}[d]^-{g} \\
 Y \ar[r]_-{\eta_Y^{\ab}} & \ab(Y) \ar[r] & B.}
$$
The top rectangle fits into the previous case, so we can conclude
that both top squares are pullbacks. As mentioned in
Section~\ref{Section S-application}, the comparison morphism
$\ab(\Eq(f)) \to \Eq(\ab(f)) $ is a regular epimorphism. Since the
top right square above is a discrete fibration, this comparison
morphism is also a (split) monomorphism, thus an isomorphism. By
applying a well-known result for regular categories---called the
``Barr-Kock Theorem'' in~\cite{Bourn-Gran-CategoricalFoundations};
see Theorem 2.17 there, or 6.10 in~\cite{Barr}---to the right hand
side diagram, we conclude that the bottom right square is a
pullback.
\end{proof}

We may conclude that all the conditions of Section~\ref{Section
S-application} are satisfied. Hence Theorem~\ref{Reflectiveness
general monoids} gives the following

\begin{theorem}
Let $\C $ be a finitely cocomplete Barr-exact unital category, and
$\Ab(\C) $ its abelian core. Then normal extensions with respect to the induced Galois structure $\Gamma_{\ab}$ are reflective
amongst regular epimorphisms.\noproof
\end{theorem}

\subsection{Monoids versus abelian groups}
We describe the normal extensions with respect to $\Gamma_{\ab}$
in the particular case when $\C $ is the category of monoids, so
that $\Ab(\C) $ is the category of abelian groups. Our description
is similar to that of Theorem~\ref{mrvdl2:Theorem: normal extensions} concerning the Galois structure $\Gamma_{\mon}$
of Section~\ref{Examples}. However, now we must add a
commutativity condition. So, we need to recall the following.

\begin{definition}[\cite{H}]
Two subobjects $x \colon X \to Z$ and $y \colon Y \to Z$ of $Z$ in
a finitely complete unital category $\C $ are said to
\defn{commute} if there exists a (necessarily unique) morphism
$\varphi \colon X \times Y \to Z$, called the \defn{cooperator} of
$x $ and $y $, such that both triangles in the diagram
\[ \xymatrix{ X \ar[r]^(.4){\langle 1_X, 0 \rangle} \ar@{{ >}->}[dr]_x & X \times
Y \ar@{.>}[d]^{\varphi} & Y \ar[l]_(.4){\langle 0, 1_Y \rangle} \ar@{{ >}->}[dl]^y \\
& Z & } \] are commutative.
\end{definition}
When two subobjects $X$ and $Y$ of $Z$ commute we write $[X,
Y]=0$. In the category of monoids, two submonoids commute if and
only if every element of the first commutes, in the usual sense,
with every element of the second.

\begin{proposition}\label{6.6}
A surjective homomorphism of monoids $f \colon A \to B$, with
kernel $k \colon N \to A$, is a normal extension with respect to
the Galois structure $\Gamma_{\ab}$ if and only if it is a special
homogeneous surjection and $[N, A]= 0$.
\end{proposition}
\begin{proof}
By definition, $f $ is a normal extension if and only if the split
epimorphism $(\pi_{1}\colon \Eq(f)\to A, \Delta\colon A\to \Eq(f))
$ is a trivial extension. By Proposition~\ref{trivial split
extensions abelian core}, this happens if and only if $N$ is an
abelian group and there exist isomorphisms $\alpha $ and~$\beta $
of split extensions as in the diagram
\begin{equation*} \label{diagram normal extensions}
\xymatrix{ N \ar@{=}[d] \ar[r]^-{\langle 1_{N}, 0 \rangle} & N
\times A \ar@<-.5ex>[r]_-{p_A} \ar@<-.5ex>[d]_{\alpha} & A
\ar@{=}[d]
\ar@<-.5ex>[l]_-{\langle 0, 1_{A} \rangle} \\
N \ar[r]_-{\langle 0, k \rangle} & \Eq(f) \ar@<-.5ex>[r]_-{\pi_1}
\ar@<-.5ex>[u]_{\beta} & A. \ar@<-.5ex>[l]_-{\Delta} }
\end{equation*}
Via Proposition~\ref{trivial split extensions abelian core}, it is
easily seen that any split epimorphic trivial extension is a
special homogeneous surjection. Then, if the surjection $f $ is a
normal extension, its kernel pair projection $\pi_1$ is a special
homogeneous surjection, and hence $f $ also is, thanks to
Proposition 7.1.5 in~\cite{SchreierBook}. Moreover, $[N, A]=0$.
Indeed, the cooperator $\varphi \colon N \times A \to A$ is given
by $\varphi = \pi_2 \comp \alpha $. Let us check that it is
actually a cooperator:
\[
\varphi \comp \langle 1_{N}, 0 \rangle = \pi_2 \comp \alpha \comp
\langle 1_{N}, 0 \rangle = \pi_2 \comp \langle 0, k \rangle = k,
\]
and
\[
\varphi \comp \langle 0, 1_{A} \rangle = \pi_2 \comp \alpha \comp
\langle 0, 1_{A} \rangle = \pi_2 \comp \Delta = 1_A.
\]

Conversely, suppose that $f $ is special homogeneous and $[N, A] =
0$. The fact that $[N, A]= 0$ defines a morphism $\alpha \colon N
\times A \to \Eq(f) $ given by $\alpha(x,a)= (a, xa) $. Let us now
describe its inverse. Since $f$ is special homogeneous, the point
$(\pi_{1}\colon {\Eq(f)\to A}, {\Delta\colon A\to \Eq(f)}) $ is a
special homogeneous split epimorphism. Using right homogeneity, we
have that for every $(a_1, a_2) \in \Eq(f) $ there exists a unique
element $q(a_1, a_2) \in N $ such that
\[
(a_1, a_2) = (1, q(a_1, a_2)) (a_1, a_1) = (a_1, q(a_1, a_2) a_1).
\]
We define a map $\beta \colon \Eq(f) \to N \times A$ by putting
$\beta(a_1, a_2) = (q(a_1, a_2), a_1) $. It is indeed the inverse
of $\alpha $, because
\[ \alpha \comp \beta(a_1, a_2) = \alpha(q(a_1, a_2), a_1) = (a_1, q(a_1, a_2)
a_1) = (a_1, a_2) \] and
\[ \beta \comp \alpha(x, a) = \beta(a, xa) = (q(a, xa), a) = (q(\langle 0, k \rangle(x)
\Delta(a)), a) = (x, a), \] where the last equality follows from
Proposition 2.1.4 in~\cite{SchreierBook}. Then $\alpha $ is an
isomorphism. It clearly is a morphism of split extensions, and
this concludes the proof.
\end{proof}

We end with a proof that, also in the case of monoids and abelian
groups, normal and central extensions coincide.

\begin{proposition}
A surjective monoid homomorphism is a normal
extension if and only if it is a central extension.
\end{proposition}
\begin{proof}
Since every normal extension is central, we only have to prove
that central extensions are normal. Let $f \colon A \to B$ be a central
extension. Then there exists a surjective morphism $p \colon E \to
B$ such that the morphism $\overline{f}$ in the pullback diagram
\[
\xymatrix{ N \ar@{=}[r] \ar[d]_{\langle 0, k \rangle} & N
\ar[d]^k \\
P \pullback \ar@{->>}[r]^{\overline{p}}
\ar@{->>}[d]_{\overline{f}} & A \ar@{->>}[d]^f
\\
E \ar@{->>}[r]_p & B }
\]
is a trivial extension. Being a trivial (and hence normal)
extension, $\overline{f}$ is a special homogeneous surjection, and
so $f $ is, thanks to Proposition 7.1.5 in~\cite{SchreierBook}.
Moreover, $[N, P] = 0$. Hence, for all $x \in N$ and all $(e,a)
\in P$, we have
\[
(1, x) (e, a) = (e, a) (1, x).
\]
Since $\overline{p}$ is surjective, this implies that $ax = xa $
for all $x \in N$ and all $a \in A$, and hence $[N, A] = 0$. This
proves that $f $ is a normal extension by Proposition~\ref{6.6}.
\end{proof}

\section*{Acknowledgements}
We are grateful to Mathieu Duckerts-Antoine, Tomas Everaert and
Sandra Man\-tovani for helpful discussions. Special thanks to the
referee for pointing out some mistakes in the first version of the
text.



\providecommand{\noopsort}[1]{}
\providecommand{\bysame}{\leavevmode\hbox to3em{\hrulefill}\thinspace}
\providecommand{\MR}{\relax\ifhmode\unskip\space\fi MR }
\providecommand{\MRhref}[2]{%
  \href{http://www.ams.org/mathscinet-getitem?mr=#1}{#2}
}
\providecommand{\href}[2]{#2}

\end{document}